\documentclass[12pt,reqno]{amsart}

\newcommand\version{October 2, 2026}

\usepackage{amsmath,amsfonts,amsthm,amssymb,amsxtra}
\usepackage{bbm} 
\usepackage{esint} 
\usepackage[hidelinks]{hyperref}

\newtheorem{theorem}{Theorem}
\newtheorem{proposition}[theorem]{Proposition}
\newtheorem{lemma}[theorem]{Lemma}

\theoremstyle{definition}

\theoremstyle{remark}

\renewcommand{\epsilon}{\varepsilon}

\newcommand{\ind}{\mathbf 1}
\newcommand{\loc}{{\rm loc}}

\renewcommand{\phi}{\varphi}
\newcommand{\R}{\mathbb{R}}

\DeclareMathOperator{\supp}{supp}

\begin{document}

\title[Another characterization of $\dot W^{1,p}(\R^d)$ --- \version]{Another characterization of $\dot W^{1,p}(\R^d)$}

\author{Rupert L. Frank}
\address[Rupert L. Frank]{Mathe\-matisches Institut, Ludwig-Maximilians Universit\"at M\"unchen, The\-resienstr.~39, 80333 M\"unchen, Germany, and Munich Center for Quantum Science and Technology, Schel\-ling\-str.~4, 80799 M\"unchen, Germany}
\email{r.frank@lmu.de}

\date{\version}
\thanks{\copyright\, 2026 by the author. This paper may be reproduced, in its entirety, for non-commercial purposes.}

\begin{abstract}
	We improve an earlier characterization of the Sobolev space $\dot W^{1,p}(\R^d)$, $1<p<\infty$, in terms of the oscillations of a function on balls of varying centers and radii. In contrast to earlier work, finiteness of the supremum is replaced by the formally weaker assumption of finiteness of the liminf, thereby answering a question by Brezis.
\end{abstract}

\maketitle

\section{Main result and discussion}

In our previous paper \cite{Fr}, we obtained a necessary and sufficient condition for the membership of a function in the homogeneous Sobolev space $\dot W^{1,p}(\R^d)$ that does not involve derivatives. Our goal here is to weaken the sufficient condition. By definition, the space $\dot W^{1,p}(\R^d)$ consists of all $f\in L^1_\loc(\R^d)$ whose distributional gradient satisfies $\nabla f\in L^p(\R^d)$. Throughout this paper, $d\geq 1$.

In order to recall the result from \cite{Fr}, we set $\R^{d+1}_+:=\R^d\times\R_+$ and, for $(a,r)\in\R^{d+1}_+$, we set $B_r(a):=\{ x\in\R^d:\ |x-a|<r\}$. For a function $f\in L^1_\loc(\R^d)$, let
$$
m_f(a,r) := \fint_{B_r(a)} \left| f(x) - \fint_{B_r(a)} f(y)\,dy \, \right|dx \,,
$$
where $\fint_{B_r(a)} \ldots = |B_r(a)|^{-1} \int_{B_r(a)}\ldots$. Finally, let $\nu_p$ be the measure on $\R^{d+1}_+$ with
$$
d\nu_p(a,r) = \frac{da\,dr}{r^{p+1}} \,.
$$

In \cite{Fr}, it was shown that, given $1<p<\infty$ and $f\in L^1_\loc(\R^d)$, we have $f\in\dot W^{1,p}(\R^d)$ if and only if 
\begin{equation}
	\label{eq:supcondition}
\sup_{\kappa>0} \kappa^p \nu_p(\{m_f>\kappa\}) < \infty \,,
\end{equation}
and
	\begin{equation}
		\label{eq:mainold}
		\| \nabla f\|_{L^p(\R^d)}^p \simeq_{p,d} \sup_{\kappa>0} \kappa^p \nu_p(\{m_f>\kappa\}) \,.
	\end{equation}
	Moreover, in this case, we have
	\begin{equation}
		\label{eq:mainlim}
		\lim_{\kappa\to 0}  \kappa^p \nu_p(\{m_f>\kappa\}) = c_{d,p}\, \| \nabla f\|_{L^p(\R^d)}^p
	\end{equation}
with an explicit constant $c_{d,p}>0$. Here and in what follows, we use the notations $\lesssim_{p,d}$, $\gtrsim_{p,d}$ and $\simeq_{p,d}$ to suppress constants that only depend on $d$ and $p$.

In this paper, we will show that in the sufficient condition \eqref{eq:supcondition} for membership $f\in\dot W^{1,p}(\R^d)$, the supremum over all $\kappa>0$ can be replaced by the liminf as $\kappa\to 0$.

The following is the main result of this paper.

\begin{theorem}\label{main}
	Let $1< p<\infty$ and let $f\in L^1_\loc(\R^d)$ with
	\begin{equation}
		\label{eq:liminfcond}
	\liminf_{\kappa\to 0} \kappa^p \nu_p(\{m_f>\kappa\}) <\infty \,.
	\end{equation}
	Then $f\in\dot W^{1,p}(\R^d)$ and
	\begin{equation}
		\label{eq:main}
		\| \nabla f\|_{L^p(\R^d)}^p \lesssim_{p,d} \liminf_{\kappa\to 0} \kappa^p \nu_p(\{m_f>\kappa\}) \,.
	\end{equation}
\end{theorem}

\medskip

\emph{Remarks.} (a) The question whether the finiteness of the liminf, rather than the supremum, is sufficient for membership in the Sobolev space was raised by Brezis in February 2022 after receiving a first version of \cite{Fr}. The present paper gives a positive answer to his question.

(b) We discuss the case $p=1$ in Section \ref{sec:p1}.

(c) Remarkably, the finiteness of the liminf (that is, \eqref{eq:liminfcond}) implies the existence of the limit (that is, \eqref{eq:mainlim}). This is a consequence of Theorem \ref{main} and the results in \cite{Fr}.

(d) There is a long history of derivative-free characterizations of Sobolev spaces. A recent review is \cite{Ng2}, which focuses, in particular, on recent contributions by Brezis and his co-authors. In particular, the work \cite{BrSevSYu} by Brezis, Seeger, Van Schaftingen and Yung unifies and extends earlier work by Nguyen \cite{Ng} (see also \cite{BoNg,BrNg,BrNg2}) and by Brezis, Van Schaftingen and Yung \cite{BrvSYu}. These works use pointwise differences rather than our oscillation function $m_f$. For previous works involving $m_f$, see \cite{Cal,DVSh,Se,RoSe2,CoSuTe} and the references in \cite{Fr}. Far-reaching extensions of the results of \cite{Fr} were obtained, for example, in \cite{ZLYYZ,HyKo,DLYYZ}. Yet other derivative-free characterizations of Sobolev spaces appear, for example, in \cite{AlMaVe,YaYuZh,DoMi}.

(e) In most of the above-mentioned works, a supremum condition (similar to \eqref{eq:supcondition}) rather than a liminf condition (similar to \eqref{eq:liminfcond}) appears. The first work that we are aware of where a liminf condition appears for weak-type threshold functionals is the short but deep work by Bourgain and Nguyen \cite{BoNg}; see also \cite{Ng1}. The sufficiency of such a liminf condition is stated as an open problem in \cite{BrSevSYu} in the context of the characterizations by Brezis, Seeger, Van Schaftingen and Yung. Partial results can be found in \cite{GoPi,Ng2} and in the references cited therein. While there are certain conceptual similarities between these proofs and ours, we think that our arguments are rather different. In particular, we do not reduce matters to a one-dimensional inequality.

\medskip

The remainder of this paper is organized as follows. In the following section, we present the proof of Theorem \ref{main}, while in Section \ref{sec:p1} we discuss the case $p=1$.


\subsection*{Acknowledgements}

The author acknowledges partial support through the German Research Foundation grants EXC-2111-390814868, TRR 352–Project-ID 470903074, and FR 2664/3-1. He is grateful to H.~Brezis for raising the question investigated in this paper. The one-dimensional version of the main result was obtained in Feburary 2022 and shared with colleagues. For the extension to the higher-dimensional case the author acknowledges the use of AI tools. All mathematical arguments and proofs in the final manuscript were checked and written by the author.


\section{Proof of Theorem \ref{main}}

\subsection{The key estimate and its consequence}

We fix $\eta\in C_c^\infty(B_1)$ with $\eta\geq0$ and $\int_{\R^d} \eta\,dx =1$ once and for all, depending only on $d$. Using the dilation convention $\eta_h(x) :=h^{-d}\eta(x/h)$ we define
\begin{equation}\label{eq:rho}
	\rho := \eta_1*\eta_{1/2}*\eta_{1/4}*\cdots.
\end{equation}
This infinite convolution is well defined and satisfies
\[
\rho\in C_c^\infty(B_2) \,,\qquad \rho\geq0\,,\qquad \int_{\R^d} \rho\,dx =1 \,.
\]
The use of such an infinite convolution is a central ingredient of our proof. It appears in a different context, for example, in \cite[Section 1.3]{Ho}.

We now come to the key estimate in this paper.

\begin{lemma}\label{lem:single}
	Let $1\leq p<\infty$. Then, for every $f\in L^1_{\loc}(\R^d)$, every $h,\kappa>0$ and every measurable $K\subset\R^d$ with finite measure, we have
	\begin{equation}\label{eq:single}
		\|\nabla(\rho_h*f)\|_{L^p(K)}
		\lesssim_{d,p}
		\frac{\kappa}{h} \, |K|^{1/p}
		+ \kappa \, \nu_p(\{m_f>\kappa\})^{1/p} \,.
	\end{equation}
\end{lemma}

Before proving this lemma, let us see how it implies our main result.

\begin{proof}[Proof of Theorem \ref{main}]
	Fix $1< p<\infty$ and $f\in L^1_{\rm loc}(\R^d)$ satisfying the liminf condition \eqref{eq:liminfcond}. For given $h>0$ and measurable $K\subset\R^d$ with finite measure, we pass to the liminf as $\kappa\to 0$ in \eqref{eq:single} and obtain
	\[
	\|\nabla(\rho_h*f)\|_{L^p(K)}
	\lesssim_{d,p} \liminf_{\kappa\to 0} \kappa\, \nu_p(\{m_f>\kappa\})^{1/p} \,.
	\]
	Exhausting $\R^d$ by bounded sets $K$ yields, by monotone convergence,
	\[
	\|\nabla(\rho_h*f)\|_{L^p(\R^d)}
	\lesssim_{d,p} \liminf_{\kappa\to 0} \kappa \, \nu_p(\{m_f>\kappa\})^{1/p} \,.
	\]
	
	Since $\rho_h*f\to f$ in $L^1_{\loc}(\R^d)$ as $h\downarrow0$, for each
	$\Phi\in C_c^1(\R^d;\R^d)$ we have
	\begin{equation}\label{eq:mainproof}
	\left|\int_{\R^d}f\,\operatorname{div}\Phi \,dx \right|
	=\lim_{h\downarrow0}\left|\int_{\R^d}\nabla(\rho_h*f)\cdot\Phi \,dx \right|
	\lesssim_{d,p} \liminf_{\kappa\to 0} \kappa\, \nu_p(\{m_f>\kappa\})^{1/p} \|\Phi\|_{L^{p'}} \,.
	\end{equation}
	Since $p>1$, duality gives a distributional gradient in $L^p$ with the asserted bound. This proves Theorem \ref{main}.
\end{proof}


\subsection{Proof of the key estimate}

\begin{proof}[Proof of Lemma \ref{lem:single}]
	We divide the proof into four steps. In the first three steps we suppose $f\in L^\infty(\R^d)$ and we remove this assumption in the fourth step. For fixed $h>0$, we put, for each $\delta>0$ and each integer $j\geq 0$,
	\[
	h_j:=2^{-j}h \,,
	\qquad 
	G_j:=\nabla(\rho_{h_j}*f),\qquad
	A_j(\delta):=\{a\in\R^d:h_j|G_j(a)|>\delta\} \,.
	\]
	
\medskip

	\emph{Step 1.} For $2h_j\leq r<4h_j$, cancellation of $\nabla\rho_{h_j}$ and its support in $B_{2h_j}$ give
	\begin{align*}
		h_j|G_j(a)|
		&=h_j\left|\int_{B_r(a)}\nabla\rho_{h_j}(a-x)
		\bigl(f(x)-f_{B_r(a)}\bigr)\,dx\right|\\
		&\leq h_j\|\nabla\rho_{h_j}\|_\infty
		\int_{B_r(a)}|f(x)-f_{B_r(a)}|\,dx \\
		& \leq C_0 \, m_f(a,r) \,,
	\end{align*}
	where $C_0$ is a constant, depending only on $d$. Hence
	\[
	A_j(\delta) \times[2h_j,4h_j)\subset\{m_f>\delta/C_0\} \,.
	\]
	We take the measure $\nu_p$ of both sides. Since the intervals $[2h_j,4h_j)$ are pairwise disjoint and
	\[
	\int_{2h_j}^{4h_j}\frac{dr}{r^{p+1}}
	=\frac{2^{-p}-4^{-p}}{p}\,h_j^{-p} \,,
	\]
	it follows that
	\begin{equation}\label{eq:count}
		\sum_{j\geq0}h_j^{-p}|A_j(\delta)|
		\leq C_p \, \nu_p\bigl(\{m_f>\delta/C_0\}\bigr) \,.
	\end{equation}
	If the measure on the right side is infinite, the estimate we seek with $\kappa = \delta/C_0$ is trivial, so we may assume that it is finite. Therefore, also the left side of \eqref{eq:count} is finite.
	
	To proceed, we need the elementary inequality
	\begin{equation}
		\label{eq:elementary}
	\left(\sum_{j\geq0}2^j b_j\right)^p
	\leq 2^p\sum_{j\geq0}2^{jp}b_j,
	\qquad b_j\in\{0,1\}.
	\end{equation}
	Indeed, if the right side is finite (otherwise the inequality is trivially true) and there is at least one nonzero $b_j$, we have $J:=\sup\{ j \geq 0:\ b_j = 1\}<\infty$ and then
	\[
	\left(\sum_{j\geq0}2^j b_j\right)^p \leq \left(\sum_{j=0}^J 2^j \right)^p = \left( 2^{J+1}-1 \right)^p < 2^{p(J+1)} \leq 2^p \sum_{j\geq 0} 2^{jp} b_j \,.
	\]
	
	For the following nonnegative function on $\R^d$,
	\[
	H(\cdot,\delta) :=\delta\sum_{j\geq0}h_j^{-1}\ind_{A_j(\delta)} \,,
	\]
	the elementary inequality \eqref{eq:elementary}, together with \eqref{eq:count}, implies
	\begin{align}\label{eq:H}
		\|H(\cdot,\delta) \|_{L^p}^p & = \frac{\delta^p}{h^p} \int_{\R^d} \left( \sum_{j\geq 0} 2^j \ind_{A_j(\delta)}(a) \right)^p da \notag \\
		& \leq 2^p\, \frac{\delta^p}{h^p} \int_{\R^d} \sum_{j\geq 0} 2^{jp} \ind_{A_j(\delta)}(a) \,da \notag \\
		& = 2^p\delta^p \sum_{j\geq0}h_j^{-p}|A_j(\delta)| \notag \\
		& \leq C_p' \, \delta^p \, \nu_p\bigl(\{m_f>\delta/C_0\}\bigr) \,.
	\end{align}
	
	\emph{Step 2.} At this point we use the following crucial relation, which directly follows from the construction of $\rho_h$,
	\begin{equation*}
		\rho_h=\eta_h*\rho_{h/2}\qquad\text{for all}\ h>0 \,.
	\end{equation*}
	It follows that 
	\begin{equation}\label{eq:gradient-refine}
		\nabla (\rho_h*f)=\eta_h*\nabla (\rho_{h/2}*f) 
		\qquad\text{for all}\ h>0 \,.
	\end{equation}
	In particular, $G_j = \eta_{h_j}* G_{j+1}$. Since $\eta_{h_j}$ is nonnegative and has unit integral, we deduce
	\[
	|G_j|\leq\eta_{h_j}*|G_{j+1}|
	\leq\eta_{h_j}*(|G_{j+1}|-h_{j+1}^{-1} \delta)_++2h_j^{-1} \delta \,.
	\]
	On the set $A_j(\delta)$ we may subtract $h_j^{-1} \delta$, while off $A_j(\delta)$ the function $(|G_{j}|-h_j^{-1} \delta)_+$ vanishes. Consequently,
	\begin{equation}\label{eq:soft}
		(|G_{j}|-h_j^{-1} \delta)_+ \leq\eta_{h_j}*(|G_{j+1}|-h_{j+1}^{-1} \delta)_+ +h_j^{-1} \delta\ind_{A_j(\delta)}.
	\end{equation}
	Let $Q_0$ denote a Dirac mass at the origin and, for $j\geq1$, let
	\[
	Q_j:=\eta_{h_0}*\eta_{h_1}*\cdots*\eta_{h_{j-1}}.
	\]
	Iterating \eqref{eq:soft} yields, for any $J\geq 1$,
	\begin{equation}\label{eq:iterate}
		(|G_{0}|-h^{-1} \delta)_+ \leq Q_J*(|G_{J}|-h_J^{-1} \delta)_+ + \sum_{j=0}^{J-1}Q_j*(h_j^{-1} \delta \ind_{A_j(\delta)}).
	\end{equation}
	For every $j\geq1$ the function $Q_j$ is nonnegative, is supported in $B_{2h}$, and satisfies $\|Q_j\|_\infty\leq\|\eta_h\|_\infty$. Thus the bound
	\begin{equation}\label{eq:kernel}
		0\leq Q_j(x)\leq C \, h^{-d} \, \ind_{B_{2h}}(x)\qquad \text{for all}\ j\geq1
	\end{equation}
	holds with a constant $C$, depending only on $d$ (and, in particular, independent of $j$).
	
\medskip
	
	\emph{Step 3.} 	Our goal is to let $J\to\infty$ in \eqref{eq:iterate}. It is at this step that we use the assumption $f\in L^\infty(\R^d)$, which gives
	\[
	\|G_J\|_\infty\leq h_J^{-1}\|\nabla\rho\|_1\|f\|_\infty.
	\]
	Since $(|G_{J}|-h_J^{-1} \delta)_+$ is bounded by $|G_J|$ and vanishes outside $A_J(\delta)$, we have, as $J\to\infty$,
	\begin{equation}\label{eq:terminal}
		\|(|G_{J}|-h_J^{-1} \delta)_+\|_1
		\lesssim_d \|f\|_\infty h_J^{-1}|A_J(\delta)|
		= \|f\|_\infty h_J^{p-1}
		\bigl(h_J^{-p}|A_J(\delta)|\bigr)\longrightarrow0 \,.
	\end{equation}
	Indeed, the terms in parentheses tend to zero by \eqref{eq:count}, and $h_J^{p-1}$ is bounded because $p\geq1$. By \eqref{eq:kernel}, $Q_J*(|G_{J}|-h_J^{-1} \delta)_+\to0$ uniformly as $J\to\infty$ for fixed $h>0$.
	
	Thus, letting $J\to\infty$ in \eqref{eq:iterate}, we find
	\begin{equation*}
		(|G_{0}|-h^{-1} \delta)_+ \leq \sum_{j=0}^{\infty}Q_j*(h_j^{-1} \delta \ind_{A_j(\delta)}).
	\end{equation*}
	Using \eqref{eq:kernel} on the terms with $j\geq1$ and recalling the definition of $H(\cdot,\delta)$ gives
	\[
	(|G_{0}|-h^{-1} \delta)_+ \leq h^{-1} \delta \ind_{A_0(\delta)}
	+C h^{-d}\ind_{B_{2h}}*H(\cdot,\delta)
	\leq H(\cdot,\delta) +C h^{-d}\ind_{B_{2h}}*H(\cdot,\delta) \,.
	\]
	By Young's convolution inequality, this pointwise inequality implies
	\[
	\|(|G_{0}|-h^{-1} \delta)_+ \|_{L^p}\leq C' \|H(\cdot,\delta) \|_{L^p}
	\]
	with a constant $C'$ depending only on $d$. Together with $|G_0|\leq h^{-1}\delta +(|G_{0}|-h^{-1} \delta)_+$ and \eqref{eq:H}, this proves
	\[
	\|\nabla(\rho_h*f)\|_{L^p(K)}
	\leq\frac{\delta}{h}|K|^{1/p}
	+C_{d,p}\left[\delta^p\nu_p\bigl(\{m_f>\delta/C_0\}\bigr)\right]^{1/p}.
	\]
	Taking $\delta=C_0\kappa$ proves the claimed inequality \eqref{eq:single} for $f\in L^\infty(\R^d)$. The constant in the result does not depend on $\|f\|_\infty$; boundedness was used only in \eqref{eq:terminal}.
	
	\medskip

	\emph{Step 4.} We now remove the boundedness assumption. Let
	\[
	T_M(s):=\max\{-M,\min\{s,M\}\} \,.
	\]
	For general $f\in L^1_{\loc}(\R^d)$, let $f_M :=T_M\circ f$. For each ball $B$,
	\begin{align*}
		\frac{1}{|B|}\int_B|f_M-(f_M)_B|\,dx \leq\frac{2}{|B|}\int_B|T_M(f)-T_M(f_B)|\,dx \leq\frac{2}{|B|}\int_B|f-f_B|\,dx \,,
	\end{align*}
	so $m_{f_M}\leq2m_f$. Applying the bounded version of \eqref{eq:single} to $f_M$ with $\kappa$ replaced by $2\kappa$, we obtain
	\[
	\|\nabla(\rho_h*f_M)\|_{L^p(K)}
	\leq2C_{d,p}\left(\frac{\kappa}{h} \, |K|^{1/p}
	+ \kappa \nu_p(\{m_f>\kappa\})^{1/p}\right).
	\]
	For fixed $h$, $L^1_{\rm loc}$ convergence $f_M\to f$ implies pointwise convergence
	$\nabla(\rho_h*f_M)\to\nabla(\rho_h*f)$. Fatou's lemma completes the proof after enlarging the constant.
\end{proof}


\section{The case \texorpdfstring{$p=1$}{p=1}}\label{sec:p1}

In this final section we examine the endpoint case $p=1$. The relevant space is $\dot{BV}(\R^d)$, which, by definition, consists of all functions in $L^1_\loc(\R^d)$ whose distributional gradient is a finite Radon measure. It was noted in \cite{Fr} that $\sup_{\kappa> 0} \kappa \nu_1(\{m_f>\kappa\}) <\infty$ implies $f\in\dot{BV}(\R^d)$ (with the corresponding bound). This assertion can be strengthened in the spirit of Theorem \ref{main}.

\begin{theorem}\label{main2}
	Let $f\in L^1_\loc(\R^d)$ with
	\begin{equation}
		\label{eq:liminfcond2}
		\liminf_{\kappa\to 0} \kappa \nu_1(\{m_f>\kappa\}) <\infty \,.
	\end{equation}
	Then $f\in\dot{BV}(\R^d)$ and
	\begin{equation}
		\label{eq:main2}
			|Df|(\R^d) \lesssim_d \liminf_{\kappa\to 0} \kappa \nu_1(\{m_f>\kappa\}) \,.
	\end{equation}
\end{theorem}

\begin{proof}
	Based on the case $p=1$ of Lemma \ref{lem:single} one establishes \eqref{eq:mainproof} exactly as in the proof of Theorem \ref{main}. For $p=1$, the variational definition of total variation then gives the asserted bound.
\end{proof}

As observed in \cite{Fr}, the property $f\in\dot{BV}(\R^d)$ does \emph{not} imply $\sup_{\kappa> 0} \kappa \nu_1(\{m_f>\kappa\}) <\infty$, the characteristic function of a ball being a counterexample. The following result strengthens this assertion, by showing that the stronger property $f\in\dot W^{1,1}(\R^d)$ does not imply the weaker liminf condition \eqref{eq:liminfcond2}.

\begin{proposition}\label{prop:counterexample}
	For every $d\geq1$, there is a nonnegative function
	\[
	f\in W^{1,1}(\R^d)\cap C_c(\R^d)
	\]
	such that
	\[
	\nu_1\bigl(\{ m_f >\kappa\}\bigr)<\infty
	\qquad\text{for every }\kappa>0,
	\]
	but
	\[
	\lim_{\kappa\downarrow0}
	\kappa\,\nu_1\bigl( \{ m_f >\kappa\}\bigr)
	=+\infty.
	\]
\end{proposition}

The counterexample in the following proof is an adaptation of one in \cite{BrNg}.

\begin{proof}
	\emph{Step 1.}
	Fix $0<\alpha<1$. Choose $\chi\in C_c^\infty((-1/2,1/2))$ with
	$0\leq\chi\leq1$ and $\chi=1$ on $[-1/4,1/4]$, and define
	\begin{equation}\label{eq:cusp}
		g(t):=
		\begin{cases}
			\chi(t)\bigl(\ln(1/t)\bigr)^{-\alpha} & \text{if}\ 0<t<1/2 \,,\\
			0 &\text{otherwise} \,.
		\end{cases}
	\end{equation}
	The function $g$ is continuous, compactly supported and nonnegative.
	On $(0,1/4)$,
	\[
	g'(t)=\frac{\alpha}{t(\ln(1/t))^{\alpha+1}},
	\qquad
	\int_0^{1/4}|g'(t)|\,dt=(\ln4)^{-\alpha}<\infty.
	\]
	Away from zero the function is smooth, and $g(0+)=g(0-)=0$.
	Thus its distributional derivative has no jump contribution at zero,
	and $g\in W^{1,1}(\R)$.
	
	For $d=1$, take $f:=g$. For $d\geq2$, choose a nonnegative function
	$\psi\in C_c^\infty(\R^{d-1})$ equal to one on $B_2'\subset\R^{d-1}$ (we use a prime to denote balls in $\R^{d-1}$), and set
	\[
	f(x_1,x'):=g(x_1)\psi(x') \,.
	\]
	Then $f\in W^{1,1}(\R^d)\cap C_c(\R^d)$ and
	\[
	\|f\|_{L^1}=\|g\|_{L^1}\|\psi\|_{L^1},
	\qquad
	\|\nabla f\|_{L^1}
	\leq\|g'\|_{L^1}\|\psi\|_{L^1}
	+\|g\|_{L^1}\|\nabla\psi\|_{L^1}<\infty.
	\]
	
	In the remainder of the proof, primed variables and cutoffs are omitted when $d=1$.
	
	\medskip
	\emph{Step 2.} Now we derive a lower bound for the mean oscillation. Let $r_0:=1/8$, and consider centers $a\in\R^d$ and radii $r$ satisfying
	\begin{equation}\label{eq:centers}
		0<r<r_0,\qquad |a_1|<r/4,\qquad a'\in B_1' \,.
	\end{equation}
	The two balls
	\[
	B^-:=B_{r/8}\!\left(a-\frac r2e_1\right),
	\qquad
	B^+:=B_{r/8}\!\left(a+\frac r2e_1\right)
	\]
	are contained in $B_r(a)$, and each has volume
	$8^{-d}|B_r(a)|$. On $B^-$ one has $x_1<-r/8$, hence $f=0$.
	On $B^+$ one has
	\[
	r/8<x_1<7r/8<1/4.
	\]
	Moreover, the cutoff $\psi$ is identically one on $B_r(a)$,
	since $|x'|<1+r<2$ there. Consequently,
	\[
	f(x)\geq\bigl(\ln(8/r)\bigr)^{-\alpha}
	\qquad\text{for }x\in B^+.
	\]
	By nonnegativity,
	\[
	f_{B_r(a)}\geq8^{-d}\bigl(\ln(8/r)\bigr)^{-\alpha}.
	\]
	Integrating the deviation from this average over $B^-$ gives
	\begin{align}
		m_f(a,r)
		&\geq\frac{1}{|B_r(a)|}
		\int_{B^-}|f(x)-f_{B_r(a)}|\,dx
		=8^{-d}f_{B_r(a)}
		\geq b_d\bigl(\ln(8/r)\bigr)^{-\alpha} \,.
		\label{eq:oscillation}
	\end{align}
	with $b_d:=8^{-2d}$.
	
	\medskip
	\emph{Step 3.} Now we show that $\kappa\, \nu_1(\{m_f>\kappa\})\to +\infty$ as $\kappa\downarrow 0$. For $\kappa>0$, set
	\[
	r_\kappa:=8\exp\!\left[
	-\left(\frac{b_d}{2\kappa}\right)^{1/\alpha}
	\right].
	\]
	For all sufficiently small $\kappa>0$, one has $r_\kappa<r_0$.
	By \eqref{eq:oscillation},
	\[
	m_f(a,r)>2\kappa>\kappa
	\]
	whenever $r_\kappa<r<r_0$ and the center satisfies
	$|a_1|<r/4$, $a'\in B_1'$.
	For each such radius $r$, the set of these centers $a\in\R^d$ has measure
	$\omega_{d-1}r/2$, where $\omega_{d-1}:=|B_1'|$ if $d\geq 2$ and $\omega_0 :=1$. 
	
	 Therefore
	\begin{align*}
		\kappa\,\nu_1(\{m_f>\kappa\})
		&\geq\frac{\omega_{d-1}}2\,
		\kappa\int_{r_\kappa}^{r_0}\frac{dr}{r}\\
		&=\frac{\omega_{d-1}}2\,
		\kappa\ln\frac{r_0}{r_\kappa}\\
		&=\frac{\omega_{d-1}}2
		\left(\frac{b_d}{2}\right)^{1/\alpha}
		\kappa^{\,1-1/\alpha}
		-\frac{\omega_{d-1}}2 \left( \ln\frac8{r_0}\right) \kappa.
	\end{align*}
	Since $0<\alpha<1$, the right side tends to $+\infty$ as $\kappa\downarrow0$.
	
	\medskip
	\emph{Step 4.} Finally, we show that $\nu_1(\{m_f>\kappa\})<\infty$ for any $\kappa>0$.
	Fix $\kappa>0$. The function $f\in C_c(\R^d)$ is uniformly continuous.
	Writing
	\[
	\omega_f(s):=\sup_{|x-y|\leq s}|f(x)-f(y)| \,,
	\]
	we have
	\[
	m_f(a,r)\leq\omega_f(2r).
	\]
	Hence there is an $\varepsilon_\kappa>0$ such that
	$m_f(a,r)\leq\kappa$ for all $a\in\R^d$ whenever
	$0<r\leq \varepsilon_\kappa$. Meanwhile,
	\[
	m_f(a,r)
	\leq\frac{2}{|B_r(a)|}\int_{B_r(a)}|f(x)|\,dx
	\leq\frac{2\|f\|_{L^1}}{\omega_d r^d}.
	\]
	Thus there is an $R_\kappa>\varepsilon_\kappa$ such that
	$m_f(a,r)\leq\kappa$ for all $a\in\R^d$ whenever $r\geq R_\kappa$.
	Choose $R>0$ with $\supp f\subset B_R(0)$.
	If $r\leq R_\kappa$ and $m_f(a,r)>\kappa$, then $B_r(a)$ must
	meet $\supp f$, so $|a|\leq R+R_\kappa$. It follows that
	\[
	\nu_1(\{m_f>\kappa\})
	\leq |B_{R+R_\kappa}(0)|
	\int_{\varepsilon_\kappa}^{R_\kappa}\frac{dr}{r^2}
	<\infty.
	\]
	This completes the proof.
\end{proof}


\bibliographystyle{amsalpha}

\end{document}